\documentclass[11pt]{article}
\usepackage{amsmath,amssymb}
\usepackage{graphicx}
\usepackage[dvips]{epsfig} 
\usepackage{amsmath,bm}
\usepackage{biograph}
\usepackage{latexsym}
\usepackage[latin1]{inputenc}
\usepackage[T1]{fontenc}
\usepackage{amsfonts}
\usepackage{amsthm}
\usepackage[mathscr]{eucal}
\usepackage[francais,english]{babel}
\usepackage{verbatim}
\usepackage{appendix}
\usepackage{color}
\usepackage{cancel}

 \newcommand{\ds}{\displaystyle}

\newtheorem{theo}{Theorem}[section]
\newtheorem{lem}{Lemma}[section]
\newtheorem{corol}{Corollary}[section]
\newtheorem{prop}{Proposition}[section]
\newtheorem{remark}{Remark}[section]

\newtheorem{asu}{Assumption}

\numberwithin{equation}{section} \numberwithin{prop}{section}
\numberwithin{corol}{section} \numberwithin{theo}{section}
\numberwithin{lem}{section}

\begin{document}
\title{On the joint asymptotic distribution of the restricted estimators in multivariate regression model}
\date{}
\author{S\'ev\'erien Nkurunziza{\thanks {University of Windsor, department of Mathematics and Statistics, 401 Sunset
Avenue, Windsor, Ontario, N9B 3P4. Email: severien@uwindsor.ca}}
\and and \,\quad{ } Youzhi Yu {\thanks {University of Windsor, 401
Sunset Avenue, Windsor, Ontario, N9B 3P4. Email: yu13e@uwindsor.ca}}
} \thispagestyle{empty} \selectlanguage{english} \maketitle
\maketitle\thispagestyle{empty}
\begin{abstract} The main Theorem of Jain~{\em et al.}[Jain, K., Singh, S., and Sharma,
S.~(2011), Restricted estimation in multivariate measurement error
regression model; JMVA, {\bf 102}, 2, 264--280] is established in
its full generality.  Namely, we derive the joint asymptotic
normality of the unrestricted estimator~(UE) and the restricted
estimators of the matrix of the regression coefficients. The derived
result holds under the hypothesized restriction as well as under the
sequence of alternative restrictions. In addition, we establish
Asymptotic Distributional Risk  for the estimators and compare their
relative performance. It is established that near the restriction,
the restricted estimators~(REs) perform better than the UE. But the
REs perform worse than the unrestricted estimator when one moves far
away from the restriction. 
\end{abstract}
 \pagenumbering{arabic} \addtocounter{page}{0}

\section{Introduction}\label{sec:introd}

In this paper, we are interested in an estimation problem in
multivariate ultrastructural measurement error model with more than
one response variable.  In particular, as in Jain~{\em et
al.}~(2011), we consider the case where the regression coefficients
may satisfy some linear restriction. It is practical to use such
models in the real world if there is at least two correlated
response variables. For example, in the field of medical sciences
(see Dolby~,~1976), more than one body index is often recorded and
the interest is to relate these measurements to the amount of
different nutrients in the daily diet. Similarly, as described in
Bertsch~{\em et al.}~(1974), in the air pollution studies , the
observed chemical elements contained in the polluted air are lead,
thorium and Uranium etc. It is highly likely that the variables
involved in the study may possess some measurement errors. Following
Mardia~(1980), multivariate regression is applicable in a wide range
of situations, such as Economics (see Meeusen,~1997) and Biology
(see Mcardle,~1988). We also refer to Stevens (2012) for a
discussion about the importance of regression models in education
and social-sciences.

In this paper, we derive the asymptotic properties of the
unrestricted and the restricted estimators of the regression
coefficients in the multivariate regression models with measurement
errors, when the coefficients satisfy some restrictions. To give a
close reference, we quote Jain {\em et al.}~(2011) who derived the
unrestricted and three restricted estimators for the regression
coefficients, and derived a theorem~(see Jain~{\em et al.},
Theorem~4.1) which gives the marginal asymptotic distributions of
the estimators under the restriction.

To summarize the contribution of this paper, we generalize
Theorem~4.1 of Jain {\em et al.}~(2011) in three ways. First, we
derive the joint asymptotic distribution of the unrestricted
estimator and any member of the class of the restricted estimators
under the restriction. Second, we derive the joint asymptotic
distribution of the unrestricted estimator and any member of the
class of the restricted estimators under the sequence of local
alternative restrictions. Third, we derive the joint asymptotic
distribution between the UE and all three restricted estimators
given in Jain~{\em et al.}~(2011), under the restriction and under
the sequence of local alternative restrictions. In addition, we
establish the Asymptotic Distributional Risk (ADR) for the UE and
the ADR of any member of the class of restricted estimators. We also
compare the relative performance of the proposed estimators. In
particular, we prove that in the neighborhood of the restriction,
the restricted estimators dominate the unrestricted estimator. We
also prove that as one moves far away from the restriction, the
unrestricted estimator dominates the restricted estimators. Finally,
we generalize Proposition~A.10 and Corollary~A.2 in Chen and
Nkurunziza~(2016).

The rest of this paper is organized as follows.
Section~\ref{sec:preliminary} outlines some preliminary results
given in Jain {\em et al.}~(2011). 
In Section~\ref{sec:largesample}, we present the main results of
this paper.  More specifically, in
Subsection~\ref{sec:largesampleres}, we establish the joint
asymptotic  distribution between the unrestricted estimator~(UE) and
any member of the restricted estimators under the restriction. In
Subsection~\ref{sec:largesampleloc}, we derive the joint asymptotic
distributions between all estimators under the sequence of the local
alternative restrictions.  In Subsection~\ref{sec:ADR}, we derive
ADR for the UE and restricted estimators and in
Subsection~\ref{sec:riskanalysis}, we analyse the relative
performance of the UE and the restricted estimators. 
Finally, Section~\ref{sec:conclusion} gives some the concluding
remark of this paper, and for the convenience of the reader, some
technical results are given in the appendix.
\section{Model Specifications and preliminary results}\label{sec:preliminary}
In this section, we describe the multivariate regression model with
measurement error as well as the assumptions used in order to
establish the results of this paper. Following Jain~{\em et al.}~(2011),
we consider the multivariate
regression model given by $$\bm{Z}=\bm{D}\bm{B} + \bm{E},$$ where $\bm{Z}$ is a $n\times q$
matrix, $\bm{D}$ is a $n\times p$ matrix, $\bm{B}$ is $p\times q$ matrix of
the regression coefficients and $\bm{E}$ is a $n\times q$ matrix of error
terms. We assume that $\bm{Z}$ is observable but D is not observable and
can be observed only through $\bm{X}$ with additional measurement error
$\bm{\Delta}$ as
$$ \bm{X}=D+\Delta,$$
where $\bm{X}$ and $\bm{\Delta}$ are $n\times p$-random matrices. Further, we
suppose that $$\bm{D}=\bm{M}+\bm{\Psi},$$ where $\bm{M}$ is a $n\times p$-matrix of
fixed components and $\bm{\Psi}$ is a $n\times p$-matrix of random
components. We also suppose that some prior
information about the regression coefficient $\bm{B}$ is available. In
particular, for known matrices $\bm{R}_{1},\bm{R}_{2}$ and $\bm{\theta}$,
we suppose that
\begin{equation}\label{restriction}
\bm{R}_{1}\bm{B}\bm{R}_2=\bm{\theta},
\end{equation}
where $\bm{R}_{1}$ is $r_1\times p$  matrix, $\bm{R}_{2}$ is a
$q\times r_2$  matrix and $\bm{\theta}$ is $r_1\times r_2$  matrix. For
the interpretation of the restriction in (\ref{restriction}),
$\bm{R}_{1}$ imposes a linear restriction on the parameters of
individual equations while $\bm{R}_{2}$ imposes a linear restriction
across equations. For more details about the interpretation of this
restriction, we refer for example to Izenman~(2008), Jain {\em et
al.}~(2011) and the references therein. To introduce some notations, let
$Z_{(i)}=[z_{1i},z_{2i},...,z_{ni}]^{'}$, let
$E_{(i)}=[\epsilon_{i1},\epsilon_{i2},...\epsilon_{iq}]^{'}$, 
let $\bm{Z}=[Z_{(1)},Z_{(2)},...,Z_{(q)}]$, 
$\bm{E}=[E_{(1)},E_{(2)},...,E_{(n)}]^{'} $. 
Further, 
let $\bm{\Delta}
=[\Delta_{(1)},\Delta_{(2)},...,\Delta_{(n)}]^{'}$ 
and let $\Delta_{(j)}=[\delta_{j1},\delta_{j2},...,\delta_{jp}]^{'}$ for
$j=1,2,...,n$, let $\bm{M}=[M_{(1)},M_{(2)},...,M_{(n)}]^{'}$
with $M_{(j)}=[M_{j1},M_{(j2)},...,M_{(jp)}]^{'}$, and let
$\Psi_{(j)}=[\psi_{j1},\psi_{j2},...,\psi_{jp}]^{'}$, $j=1,2,...,n$. We also let
$\bm{I}_{p}$ to stand for the $p$-dimensional identity matrix.
The following assumptions are made in order to derive the proposed
estimators and their asymptotic properties. Note that these conditions are similar to that in Jain {\em et al.}~(2011).

\begin{asu}
\begin{description}
\item[($\mathcal{A}_1$)] Elements of vector
$E_{(i)}=[\epsilon_{1i},\epsilon_{21},...,\epsilon_{ni}]$ are
independent with mean $0$, variance ${\sigma^2_{\epsilon}}$, third
moment $\gamma_{1\epsilon}{\sigma^3_{\epsilon}}$ and fourth moment
$(\gamma_{2\epsilon}+3){\sigma^4_{\epsilon}};$
\item[($\mathcal{A}_2$)] $\delta_{ij}$ are independent and identically
distributed random variables with mean $0$, variance
${\sigma_\delta}^2$, third moment $\gamma_{1\delta}
{\sigma_\delta}^3$ and fourth moment
$(\gamma_{2\delta}+3){\sigma_\delta}^4;$
\item[($\mathcal{A}_3$)] $\Psi_{ij}$ are independent and identically
distributed random variables with mean 0, variance
${\sigma_\varPsi}^2$, third moment
$\gamma_{1\varPsi}{\sigma_\varPsi}^3$ and fourth moment
$(\gamma_{2\varPsi}+3){\sigma_\varPsi}^4$;
\item[($\mathcal{A}_4$)] $\bm{\Delta}$,$\bm{\Psi}$, and $\bm{E}$ are mutually independent;
\item[($\mathcal{A}_5$)] $M_{(n)}\rightarrow\sigma_M$ as $n\rightarrow
\infty$ and $\sigma_M$ is finite;
\item[($\mathcal{A}_6$)] \textrm{\em Rank}($X$)=$p$, \textrm{\em Rank}($\bm{R}_{1}$)=$r_1$
and \textrm{\em Rank}($\bm{R}_{2}$)=$r_2$.
\end{description}
\end{asu}
\subsection{ Estimation methods}\label{sec:UE}
In this subsection, we outline some results given in Jain {\em et
al.}~(2011) which are used to derive the main results of this paper.
Namely, we present the unrestricted estimator~(UE) and three
restricted estimators~(REs) of the regression coefficients. By using
the class of objective functions given in Jain~{\em et al.}~(2011),
we also present a class of the restricted estimators which includes
the three REs. For more details about the content of this
subsection, we refer to Jain~{\em et al.}~(2011).

\subsubsection{The unrestricted estimator}

As in Jain~{\em et al.}~(2011), one considers first the following
objective function\\
$\mathcal{G}_1=\text{tr}((\bm{Z}-\bm{X}\bm{B})^{'}(\bm{Z}-\bm{X}\bm{B}))$,
which leads to the least squares estimators (LSE)
\begin{eqnarray}
\hat{\bm{B}}=(\bm{X}^{'}\bm{X})^{-1}\bm{X}^{'}\bm{Z}.\label{LSE}
\end{eqnarray}
Under parts $(\mathcal{A}_1)-(\mathcal{A}_5)$ of Assumption~1, one
can verify that $\hat{\bm{B}}$ converges in probability to
$\bm{K}\bm{B}\neq \bm{B}$, where $\bm{K}=\bm{\Sigma}^{-1}
\left(\bm{\Sigma}-\sigma_{\delta}^{2}\bm{I}_{p}\right)$ with
$\bm{\Sigma}=\sigma_M\sigma^{'}_M+\sigma^2_\psi
\bm{I}_p+\sigma^2_\delta \bm{I}_p$, and thus, $\hat{\bm{B}}$ is not
a consistent estimator. 
Because of that, as in Jain {\em et al.}~(2011), one replaces $\hat{\bm{B}}$ by 
\begin{eqnarray}
\hat{\bm{B}}_1={K^{-1}_X}\hat{\bm{B}},\label{B1hat}
\end{eqnarray}
where 
$\bm{K}_{X}=\bm{\Sigma}_{X}^{-1}\bm{\Sigma}_{D}$ with
$\bm{\Sigma}_{X}=n^{-1}\,\bm{M}'\bm{M}+\sigma^{2}_{\psi}\bm{I}_{p}+\sigma^{2}_{\delta}\bm{I}_{p}$,
$\bm{\Sigma}_{D}=n^{-1}\,\bm{M}'\bm{M}+\sigma^{2}_{\psi}\bm{I}_{p}$. %
Further, as in Jain~{\em et al.}~(2011), one can verify that
$$\bm{\Sigma}_X\xrightarrow[n\rightarrow\infty]{p}\bm{\Sigma}, \quad{ } \mbox{ and } \quad{ } \bm{\Sigma}_D\xrightarrow[n\rightarrow\infty]{p}\bm{\Sigma}-\sigma^2_\delta \bm{I}_p \quad{ }  \mbox{ where }\, \bm{\Sigma}=[\sigma_M{\sigma^{'}_M}+{\sigma^{2}_\psi}\bm{I}_p+{\sigma_\delta}^{2}\bm{I}_p].$$
As given in Jain {\em et al.}~(2011), note that the estimator
$\hat{\bm{B}}_{1}$ can be obtained directly by minimizing the
objective function
\begin{eqnarray}
\hat{\mathcal{G}}_2=\mathcal{G}_1-\text{tr}[\bm{B}^{'}(n\bm{\Sigma}_{X})(\bm{I}_p-\bm{K}_{X})\bm{B}].\label{G2hat}
\end{eqnarray}
 For more details, we refer to Jain {\em et al.}~(2011). In the
 quoted paper, the authors prove that $\bm{B}_{1}$ is a consistent
estimator for $\bm{B}$.
They also derive the following theorem which gives the asymptotic
distribution of $\sqrt{n}(\hat{\bm{B}}_{1}-\bm{B})$. To introduce
some notations, let\\
$\bar{K}_X=\left(\bm{\Sigma}_X-\bm{\Sigma}_D\right)\bm{\Sigma}^{-1}_X$,
$H=n^{-\frac{1}{2}}\bm{X}^{'}\bm{X}-n^{\frac{1}{2}}\bm{\Sigma}_X$,
$h=n^{-\frac{1}{2}}(\bm{X}^{'}[\bm{E}-\bm{\Delta}
\bm{B}])+n^{\frac{1}{2}}{\sigma_\delta}^{2}\bm{B}$,
$\bm{\Lambda}=\displaystyle\lim_{n\rightarrow
\infty}\textrm{E}\{[\text{vec}(h^{'})+\text{vec}(\bm{B}^{'}\bar{K_X}H)]{[\text{vec}(h^{'})+\text{vec}(\bm{B}^{'}\bar{K_X}H)]}^{'}\}$
and let $\bm{0}$ be a zero-matrix. The existence of this matrix is
established in Jain~{\em et al.}~(2011).

\begin{theo}\label{theorhH}
Suppose that Assumptions ($\mathcal{A}_1$)-($\mathcal{A}_6$) hold
hold, we have\\
$n^{\frac{1}{2}}(\hat{\bm{B}}_1-\bm{B})\xrightarrow[n\rightarrow\infty]{d}\eta_1\sim
\mathcal{N}_{p\times q}(\bm{0}, \bm{A}_1\bm{\Lambda}
 \bm{A}'_{1})$, where $\bm{A}_1=(\bm{\Sigma} \bm{K})^{-1}\otimes \bm{I}_q$.
\end{theo}

The proof is similar to that given in Jain {et al.}~(2011, see the
proof of Theorem~4.1). 
\subsubsection{A class of restricted estimators}
In this subsection, we present a class of  estimators of $\bm{B}$
which are consistent and satisfy the restriction
in~\eqref{restriction}. As commonly the case in constrained
estimation, this is obtained by minimizing a certain objective
function subject to the constraint. In particular, since the
objective function $\hat{\mathcal{G}}_{2}$ given in~\eqref{G2hat}
leads to a consistent estimator, the RE can be obtained by
minimizing $\hat{\mathcal{G}}_{2}$ subject to the constraint
$\bm{R}_{1}\bm{B}\bm{R}_2=\bm{\theta}$. The following proposition
shows that the above objective function can be seen as a member of a
certain class of objective functions. For more details, we refer to
Jain~{\em et al.}~(2011).
\begin{prop} \label{propG2hat} We have
$\hat{\mathcal{G}}_2=\text{\em tr}(\bm{Z}^{'}\bm{Z})+\text{\em
tr}[(\hat{\bm{B}}_1-\bm{B})^{'}(\bm{X}^{'}\bm{X})\bm{K}_X(\hat{\bm{B}}_1-\bm{B})]$.
\end{prop}
The proof follows directly from algebraic computations. 
From Proposition~\ref{propG2hat}, as in
Jain~{\em et al.}~(2011) one considers below a more general class of
objective functions. To this end, let $\mathcal{P}_{p\times p}$
denote the set of all observable $p\times p$-symmetric and positive
definite matrices and let
\begin{eqnarray}\label{classofobject}
\left\{\hat{\mathcal{G}}_{3}\left(\hat{\bm{\Sigma}}\right)
=\text{tr}((\hat{\bm{B}}_1-\bm{B})^{'}\hat{\bm{\Sigma}}(\hat{\bm{B}}_1-\bm{B})):
\hat{\bm{\Sigma}} \in \mathcal{P}_{p\times p}\right\}.
\end{eqnarray}
Thus, $\hat{\mathcal{G}}_2$ is a member of this class with
$\hat{\bm{\Sigma}}=(\bm{X}^{'}\bm{X})\bm{K}_X$. Other members of
objective functions correspond to the cases where
$\hat{\bm{\Sigma}}=\bm{S}=\bm{X}^{'}\bm{X}$ and
$\hat{\bm{\Sigma}}=n\,\bm{I}_{p}$. For further details about the
objective function in~\eqref{classofobject}, we refer to Jain~{\em
et al.}~(2011). From the above class of objective function, one
obtains a class of restricted estimators
$\{\tilde{B}(\hat{\bm{\Sigma}}): \hat{\bm{\Sigma}} \in
\mathcal{P}_{p\times p}\}$ which satisfies the constraint
$\bm{R}_{1}\bm{B}\bm{R}_2=\theta$. Namely, by using the Lagrangian
method, we get
\begin{eqnarray}\label{BtildeofW}
\tilde{B}(\hat{\bm{\Sigma}})=\hat{\bm{B}}_1-(\hat{\bm{\Sigma}})^{-1}\bm{R}^{'}_1\left[\bm{R}_{1}(\hat{\bm{\Sigma}})^{-1}\bm{R}_{1}\right]^{-1}\left(\bm{R}_{1}\hat{\bm{B}}_1\bm{R}_{2}-\theta\right)(\bm{R}^{'}_{2}\bm{R}_{2})^{-1}\bm{R}^{'}_{2},
\end{eqnarray}
where $\hat{\bm{\Sigma}}$ is a known symmetric and positive definite
matrix. In particular, from \eqref{BtildeofW}, by replacing
$\hat{\bm{\Sigma}}$ by $(\bm{X}^{'}\bm{X})\bm{K}_X$,
$\bm{X}^{'}\bm{X}$ and $n\bm{I}_{p}$, respectively, one gets
\begin{eqnarray}\label{B2hat}
\hat{\bm{B}}_2=\hat{\bm{B}}_1-(\bm{X}^{'}\bm{X} \bm{K}_X
)^{-1}\bm{R}^{'}_1[\bm{R}_{1}(\bm{X}^{'}\bm{X}
\bm{K}_X)^{-1}\bm{R}^{'}_1]^{-1}(\bm{R}_{1}\hat{\bm{B}}_1\bm{R}_{2}-\theta)({\bm{R}^{'}_{2}}\bm{R}_{2})^{-1}{\bm{R}^{'}_{2}},\,
\, \, \,
\end{eqnarray}
\begin{equation}\label{B3}
\hat{\bm{B}}_3=\hat{\bm{B}}_1-(\bm{X}^{'}\bm{X})^{-1}\bm{R}^{'}_1\left[\bm{R}_{1}(\bm{X}^{'}\bm{X})^{-1}\bm{R}^{'}_1\right]^{-1}\left(\bm{R}_{1}\hat{\bm{B}}_1\bm{R}_{2}-\theta\right)({\bm{R}^{'}_{2}}\bm{R}_{2})^{-1}{\bm{R}^{'}_{2}},
\end{equation}
\begin{equation}\label{B4}
\hat{\bm{B}}_4=\hat{\bm{B}}_1-\bm{R}^{'}_1\left(\bm{R}_{1}
\bm{R}^{'}_1\right)^{-1}\left(\bm{R}_{1}\hat{\bm{B}}_1\bm{R}_{2}-\theta\right)({\bm{R}^{'}_{2}}\bm{R}_{2})^{-1}{\bm{R}^{'}_{2}}.
\end{equation}
Note that the estimators $\hat{\bm{B}}_2$, $\hat{\bm{B}}_3$ and
$\hat{\bm{B}}_4$ are derived in Jain {\em et al.}~(2011). Here,
their derivation is given for the paper to be self-contained.

\section{Main results }\label{sec:largesample}
In this section, we derive the joint asymptotic distribution of all
estimators, under the restriction as well as under the sequence of
local alternative restrictions. In particular, we generalize
Theorem~4.1 in Jain~{\em et al.}~(2011) which gives the marginal
asymptotic distributions under the restriction.
\subsection{Asymptotic properties under the restriction}\label{sec:largesampleres}
In this subsection, we derive the joint asymptotic normality of the
UE and any member of the restricted estimators, under the
restriction. We suppose that the weighting matrix
$\hat{\bm{\Sigma}}$ satisfies the following assumption.
\begin{asu}\label{assumption:W*}
 $\hat{\bm{\Sigma}}$ is such that $\ds{\frac{1}{n}}\hat{\bm{\Sigma}}\xrightarrow[n\rightarrow\infty]{P}Q_{0}$ where $Q_{0}$ is nonrandom and positive definite matrix.
\end{asu}
Note that the matrices $\bm{X}^{'}\bm{X} K_X$, $\bm{X}^{'}\bm{X}$
and $n \bm{I}_{p}$ satisfy Assumption~2 with the matrix $Q_{0}$
equals to $\bm{\Sigma}\bm{K}$, $\bm{\Sigma}$ and $\bm{I}_{p}$
respectively. To set up some notations, let
$A_{1}=(\bm{\Sigma}\bm{K})^{-1}\otimes \bm{I}_{q}$, let
\begin{eqnarray}
& &
A_2\left(\bm{Q}_{0}\right)=A_{1}-(\bm{Q}_{0})^{-1}\bm{R}^{'}_1(\bm{R}_{1}(\bm{Q}_{0})^{-1}\bm{R}^{'}_1)^{-1}\bm{R}_{1}
\left(\bm{\Sigma}\bm{K}\right)^{-1}\otimes
\bm{R}_{2}(R'_{2}\bm{R}_{2})^{-1}R'_{2},\nonumber\\
& & \bm{\Sigma}_{11}=\bm{A}_1\bm{\Lambda}  \bm{A}'_{1},\,
\,\bm{\Sigma}_{22}\left(\bm{Q}_{0}\right)=\bm{A}_2\left(\bm{Q}_{0}\right)\bm{\Lambda}
\bm{A}^{'}_2\left(\bm{Q}_{0}\right), \,\,
\bm{\Sigma}_{21}\left(\bm{Q}_{0}\right)=\bm{\Sigma}_{12}^{'}\left(\bm{Q}_{0}\right),\, \label{Sigmaij} \\
& & \bm{\Sigma}_{12}\left(\bm{Q}_{0}\right)=\bm{A}_1\bm{\Lambda}
\bm{A}^{'}_2\left(\bm{Q}_{0}\right).\nonumber
\end{eqnarray}
\begin{theo}\label{propjointclass}
 If Assumptions 1-2 hold and  $\bm{R}_{1}\bm{B}\bm{R}_2=\theta$, we have\\
 $\left(n^\frac{1}{2}\left(\hat{\bm{B}}_{1}-\bm{B}\right)',\,n^\frac{1}{2}\left(\tilde{B}(\hat{\bm{\Sigma}})-\bm{B}\right)'\right)'
 \xrightarrow[n\rightarrow\infty]{d}\left(\eta'_{1},\,\zeta^{*'}\right)'$
 where
 \begin{eqnarray}
\left(
  \begin{array}{c}
    \eta_{1} \\
    \zeta^{*} \\
  \end{array}
\right)\sim \mathcal{N}_{2p\times q}\left(\left(
                                        \begin{array}{c}
                                          \bm{0} \\
                                          \bm{0} \\
                                        \end{array}
                                      \right), \left(
                                                 \begin{array}{cc}
                                                   \bm{\Sigma}_{11} & \bm{\Sigma}_{12}\left(\bm{Q}_{0}\right) \\
                                                   \bm{\Sigma}_{21}\left(\bm{Q}_{0}\right) & \bm{\Sigma}_{22}\left(\bm{Q}_{0}\right) \\
                                                 \end{array}
                                               \right)
\right),
\end{eqnarray}
where $\bm{\Sigma}_{11}$, $\bm{\Sigma}_{12}\left(\bm{Q}_{0}\right)$,
$\bm{\Sigma}_{21}\left(\bm{Q}_{0}\right)$ and
$\bm{\Sigma}_{22}\left(\bm{Q}_{0}\right)$ are defined
in~\eqref{Sigmaij}.
\end{theo}
The proof of this theorem is given in the Appendix. The above
theorem generalizes Theorem~4.1 in Jain {\em et al.}~(2011) in two
ways. First, the estimator $\hat{\bm{B}}(\hat{\bm{\Sigma}})$
encloses as special cases the restricted estimators
$\hat{\bm{B}}_{2}$, $\hat{\bm{B}}_{3}$ and $\hat{\bm{B}}_{4}$.
Second, the above result gives the joint asymptotic distribution
between the UE and any member of the class of restricted estimators;
from which the marginal asymptotic distribution follows directly.
Indeed, if $\bm{Q}_{0}$ is taken as $K\bm{\Sigma}$, $\bm{\Sigma}$
and $\bm{I}_{p}$, respectively, the above result gives the
asymptotic distribution of
$n^{1/2}\left(\hat{\bm{B}}_{2}-\bm{B}\right)$,
$n^{1/2}\left(\hat{\bm{B}}_{3}-\bm{B}\right)$, and
$n^{1/2}\left(\hat{\bm{B}}_{4}-\bm{B}\right)$  given in Jain {\em et
al.}~(2011). Below, we give another generalization of the limiting
distributions given in Jain {\em et al.}~(2011). In particular, we
establish the joint asymptotic normality between the estimators
$\hat{\bm{B}}_{1}$, $\hat{\bm{B}}_{2}$, $\hat{\bm{B}}_{3}$ and
$\hat{\bm{B}}_{4}$, under the sequence of local alternative
restrictions. On the top of this result, as intermediate step,  we
also generalize Proposition~A.10 and Corollary~A.2 in Chen and
Nkurunziza~(2016).

\subsection{Asymptotic results under local
alternative}\label{sec:largesampleloc}
In this subsection, we present the asymptotic properties of the UE
and the restricted estimators under the following sequence of local
alternative restrictions
\begin{equation}
\bm{R}_{1}\bm{B}\bm{R}_2=\bm{\theta}+\frac{\bm{\theta}_0}{\sqrt{n}},\ \ \ \
n=1,2,.... \label{localaltern}
\end{equation} where $\bm{\theta}_0$ is fixed with $||\bm{\theta}_0||<\infty$.
 Note that if $\bm{\theta}_0=0$ in
(\ref{localaltern}), then (\ref{localaltern}) becomes
(\ref{restriction}). Thus, the results established under
(\ref{localaltern}) generalize the results given in Jain~{\em et
al.}~(2011), which are established under (\ref{restriction}).
\begin{theo}\label{propjointclassloc}
Suppose that Assumptions 1 and 2 hold along with the sequence of
local alternative in~\eqref{localaltern}, then
 $n^\frac{1}{2}\left(\left(\hat{\bm{B}}_{1}-\bm{B}\right)',\,\left(\tilde{B}(\hat{\bm{\Sigma}})-\bm{B}\right)'\right)'
 \xrightarrow[n\rightarrow\infty]{d}\left(\eta'_{1},\,\eta^{*'}\right)'$
 where
 \begin{eqnarray}
\left(
  \begin{array}{c}
    \eta_{1} \\
    \eta^{*} \\
  \end{array}
\right)\sim \mathcal{N}_{2p\times q}\left(\left(
                                        \begin{array}{c}
                                          \bm{0} \\
                                          \bm{\mu}\left(\bm{Q}_{0}\right) \\
                                        \end{array}
                                      \right), \left(
                                                 \begin{array}{cc}
                                                   \bm{\Sigma}_{11} & \bm{\Sigma}_{12}\left(\bm{Q}_{0}\right) \\
                                                   \bm{\Sigma}_{21}\left(\bm{Q}_{0}\right) & \bm{\Sigma}_{22}\left(\bm{Q}_{0}\right) \\
                                                 \end{array}
                                               \right)
\right),
\end{eqnarray}
where
$\bm{\mu}\left(\bm{Q}_{0}\right)=-\bm{Q}_{0}^{-1}\bm{R}^{'}_1(\bm{R}_{1}\bm{Q}_{0}^{-1}\bm{R}^{'}_1)^{-1}
\theta_0(\bm{R}^{'}_{2}\bm{R}_{2})^{-1}\bm{R}^{'}_{2},$,
$\bm{\Sigma}_{11}$, $\bm{\Sigma}_{12}\left(\bm{Q}_{0}\right)$,
$\bm{\Sigma}_{21}\left(\bm{Q}_{0}\right)$ and
$\bm{\Sigma}_{12}\left(\bm{Q}_{0}\right)$  are defined as in
Theorem~\ref{propjointclass}.
\end{theo}
The proof of this theorem is given in the Appendix. By using the
similar techniques, we establish the joint distribution of the UE
and the restricted estimators given in \eqref{B2hat}, \eqref{B3} and
\eqref{B4}. To introduce some notations, let
\begin{eqnarray}
\bm{A}_{2}=\bm{A}_{1}-(\bm{\Sigma}\bm{K})^{-1}\bm{R}^{'}_1(\bm{R}_{1}(\bm{\Sigma}\bm{K})^{-1}\bm{R}^{'}_1)^{-1}\bm{R}_{1}
\left(\bm{\Sigma}\bm{K}\right)^{-1}\otimes
\bm{R}_{2}(R'_{2}\bm{R}_{2})^{-1}R'_{2},\nonumber\\
 \bm{A}_{3}=\bm{A}_{1}-(\bm{\Sigma})^{-1}\bm{R}^{'}_1(\bm{R}_{1}(\bm{\Sigma})^{-1}\bm{R}^{'}_1)^{-1}\bm{R}_{1}
\left(\bm{\Sigma}\bm{K}\right)^{-1}\otimes
\bm{R}_{2}(R'_{2}\bm{R}_{2})^{-1}R'_{2},\nonumber\\
\bm{A}_{4}=\bm{A}_{1}-\bm{R}^{'}_1(\bm{R}_{1}
\bm{R}^{'}_1)^{-1}\bm{R}_{1}
\left(\bm{\Sigma}\bm{K}\right)^{-1}\otimes
\bm{R}_{2}(R'_{2}\bm{R}_{2})^{-1}R'_{2},\nonumber\\
 \bm{\Sigma}_{ij}=\bm{A}_{i}\bm{\Lambda} \bm{A}^{'}_{j},
i=1,2,3,4, j=1,2,3,4, \nonumber\\
\bm{\mu}_{2}=-(\Sigma
K)^{-1}\bm{R}^{'}_1(\bm{R}_{1}(\Sigma
K)^{-1}\bm{R}^{'}_1)^{-1}\theta_0(\bm{R}^{'}_{2}\bm{R}_{2})^{-1}\bm{R}^{'}_{2},\nonumber\\
\bm{\mu}_{3}=-\Sigma^{-1}\bm{R}^{'}_1(\bm{R}_{1}\Sigma^{-1}\bm{R}^{'}_1)^{-1}\theta_0(\bm{R}^{'}_{2}\bm{R}_{2})^{-1}\bm{R}^{'}_{2},
\bm{\mu}_{4}=-\bm{R}^{'}_1(\bm{R}_{1}\bm{R}^{'}_1)^{-1}\theta_0(\bm{R}^{'}_{2}\bm{R}_{2})^{-1}\bm{R}^{'}_{2}.\nonumber
\end{eqnarray}
\begin{theo}\label{propjoinnorm}
  If Assumption~1 holds along with~\eqref{localaltern}, we have\\
   $\left(n^\frac{1}{2}\left(\hat{\bm{B}}_{1}-\bm{B}\right)',\,n^\frac{1}{2}\left(\hat{\bm{B}}_{2}-\bm{B}\right)',
   \,n^\frac{1}{2}\left(\hat{\bm{B}}_{3}-\bm{B}\right)',\,n^\frac{1}{2}\left(\hat{\bm{B}}_{4}-\bm{B}\right)'\right)'
 \xrightarrow[n\rightarrow\infty]{d}\bm{\eta}$
 where
 \begin{eqnarray}
\bm{\eta}=\left(
  \begin{array}{c}
    \bm{\eta}_{1} \\
    \bm{\eta}_{2} \\
   \bm{\eta}_{3} \\
    \bm{\eta}_{4}
  \end{array}
\right)\sim \mathcal{N}_{4p\times q}\left(\left(
                                        \begin{array}{c}
                                          \bm{0} \\
                                          \bm{\mu}_{2} \\
                                          \bm{\mu}_{3} \\
                                          \bm{\mu}_{4} \\
                                        \end{array}
                                      \right),\,\left(
                                                  \begin{array}{cccc}
                                                    \bm{\Sigma}_{11} & \bm{\Sigma}_{12} & \bm{\Sigma}_{13} & \bm{\Sigma}_{14} \\
                                                    \bm{\Sigma}_{21} & \bm{\Sigma}_{22} & \bm{\Sigma}_{23} & \bm{\Sigma}_{24} \\
                                                    \bm{\Sigma}_{31} & \bm{\Sigma}_{32} & \bm{\Sigma}_{33} & \bm{\Sigma}_{34} \\
                                                    \bm{\Sigma}_{41} & \bm{\Sigma}_{32} & \bm{\Sigma}_{43} & \bm{\Sigma}_{44} \\
                                                  \end{array}
                                                \right)
\right).
\end{eqnarray}
\end{theo}
The proof of this theorem is given in the Appendix. Since the
sequence of local alternative includes as a special case the
restriction, one deduces the following corollary.
\begin{corol}
  If Assumption 1 holds and  $\bm{R}_{1}\bm{B}\bm{R}_2=\bm{\theta}$, we have\\
   $n^\frac{1}{2}\left(\left(\hat{\bm{B}}_{1}-\bm{B}\right)',\,\left(\hat{\bm{B}}_{2}-\bm{B}\right)',
   \,\left(\hat{\bm{B}}_{3}-\bm{B}\right)',\,\left(\hat{\bm{B}}_{4}-\bm{B}\right)'\right)'
 \xrightarrow[n\rightarrow\infty]{d}\left(\bm{\eta}'_{1},\,\bm{\zeta}'_{2},\,\bm{\zeta}'_{3},\,\bm{\zeta}'_{4}\right)'$
 where
 \begin{eqnarray}
\left(
  \begin{array}{c}
    \bm{\eta}_{1} \\
    \bm{\zeta}_{2} \\
    \bm{\zeta}_{3} \\
    \bm{\zeta}_{4}
  \end{array}
\right)\sim \mathcal{N}_{4p\times q}\left(\left(
                                        \begin{array}{c}
                                          \bm{0} \\
                                          \bm{0} \\
                                          \bm{0} \\
                                          \bm{0} \\
                                        \end{array}
                                      \right),\,\left(
                                                  \begin{array}{cccc}
                                                    \bm{\Sigma}_{11} & \bm{\Sigma}_{12} & \bm{\Sigma}_{13} & \bm{\Sigma}_{14} \\
                                                    \bm{\Sigma}_{21} & \bm{\Sigma}_{22} & \bm{\Sigma}_{23} & \bm{\Sigma}_{24} \\
                                                    \bm{\Sigma}_{31} & \bm{\Sigma}_{32} & \bm{\Sigma}_{33} & \bm{\Sigma}_{34} \\
                                                    \bm{\Sigma}_{41} & \bm{\Sigma}_{32} & \bm{\Sigma}_{43} & \bm{\Sigma}_{44} \\
                                                  \end{array}
                                                \right)
\right).
\end{eqnarray}
\end{corol}
The proof follows directly from Theorem~\ref{propjoinnorm} by taking
$\bm{\theta}_{0}=0$.
\subsection{ Asymptotic Distributional Risk}
\label{sec:ADR} Asymptotic Distributional Risk (ADR) is one of the
important statistical tools to compare different estimators. In this
subsection, we derive ADR of the UE and that of any member of the
proposed class of the restricted estimators, i.e. ADR of $\hat{\bm{B}}_{1}$ and
$\tilde{\bm{B}}(\hat{\bm{\Sigma}})$. Recall that, if
$\sqrt{n}(\hat{\bm{\theta}}-\bm{\theta})\xrightarrow[n\rightarrow
\infty]{d}\bm{U}$, where $\hat{\bm{\theta}}$, $\bm{\theta}$ and
$\bm{U}$ are matrices.
The ADR is defined as\\
$\text{ADR}(\hat{\theta},\theta;W)=\text{E}[\text{tr}(U^{'}WU)]$,
where $W$ is a weighting matrix. For more details about the ADR, we
refer for example to Saleh (2006), Chen and Nkurunziza~(2015, 2016)
and references therein. To introduce some notations, let
$\bm{C}_3(\bm{Q}_{0})=\bm{Q}_{0}^{-1}\bm{R}^{'}_1(\bm{R}_{1}\bm{Q}_{0}^{-1}\bm{R}^{'}_1)^{-1}$,
$\bm{C}_4=(\bm{R}^{'}_{2}\bm{R}_{2})^{-1}\bm{R}^{'}_{2}$,
$\bm{J}_1(\bm{Q}_{0})=\bm{C}_3(\bm{Q}_{0})\bm{R}_{1}\bm{Q}_{0}^{-1}$
and $\bm{J}=\bm{R}_{2}\bm{C}_4$.

\begin{theo}\label{propADR}
Suppose that the conditions of Theorem~\ref{propjointclassloc} hold,
then
\begin{eqnarray*}
& &\text{\em ADR}(\hat{\bm{B}}_1,\bm{B};W)=\text{\em tr}((W \otimes \bm{I}_{q})(A_1\Lambda  \bm{A}'_{1})),\\
& & \text{\em ADR}(\tilde{\bm{B}}(\hat{\bm{\Sigma}}),\bm{B},W) =\text{\em
ADR}(\hat{\bm{B}}_1,\bm{B},W)-f_1(\bm{Q}_{0})+(\text{\em
vec}(\theta_0))^{'}F_1(\bm{Q}_{0})\text{\em vec}(\theta_0),
\end{eqnarray*}
with \quad{ } $f_{1}(\bm{Q}_{0})=\text{\em
tr}(((J^{'}_1(\bm{Q}_{0})W\otimes J)\otimes I_{pq})\text{\em
vec}(((\Sigma K)^{-1}\otimes
I_q))(\text{\em vec}(\Lambda))^{'})$ \quad{ } and\\
$F_{1}(\bm{Q}_{0})= C^{'}_3(\bm{Q}_{0})WC_3(\bm{Q}_{0})\otimes
(\bm{R}^{'}_{2}\bm{R}_{2})^{-1}$.
\end{theo}
The proof of this theorem follows from Theorem~\ref{propjointclass}.
For the convenience of the reader, it is also outlined in the
Appendix.
\subsection{Risk
Analysis}\label{sec:riskanalysis} In this section, we compare
$\text{ADR}(\tilde{\bm{B}}(\hat{\bm{\Sigma}}),\bm{B},W)$ and
$\text{ADR}(\hat{\bm{B}}_1,\bm{B};W)$ in order to evaluate the
relative performance of $\tilde{\bm{B}}(\hat{\bm{\Sigma}})$ and
$\hat{\bm{B}}_1$. To simply some notations, for a given symmetric
matrix $A$, let $\textrm{ch}_{\min}(A)$ and $\textrm{ch}_{\max}(A)$
be, respectively, the smallest and largest eigenvalues of $A$.
\begin{theo}\label{compareADR}
Suppose that the conditions of Theorem~\ref{propADR} hold. If
$||\theta_0||^2< \displaystyle\frac{f_1(\bm{Q}_{0})}{\textrm{\em
ch}_{max}(F_1(\bm{Q}_{0}))}$, then $\text{\em
ADR}(\tilde{\bm{B}}(\hat{\bm{\Sigma}}),\bm{B},W)\leq \text{\em
ADR}(\hat{\bm{B}}_1,\bm{B};W)$.\quad{ }  If
$||\theta_0||^2>\displaystyle\frac{f_1(\bm{Q}_{0})}{\textrm{\em
ch}_{min}(F_1(\bm{Q}_{0}))}$,\\ then $\text{\em
ADR}(\tilde{\bm{B}}(\hat{\bm{\Sigma}}),\bm{B},W)>\text{\em
ADR}(\hat{\bm{B}}_1,\bm{B};W)$.
\end{theo}
The proof of this theorem is given in the Appendix.
\begin{remark}
Since $\text{\em ADR}(\tilde{\bm{B}}(\hat{\bm{\Sigma}}),\bm{B},W)$
and  $\text{\em ADR}(\hat{\bm{B}}_1,\bm{B};W)$ are positive real
numbers,\\ $\text{\em
ADR}(\tilde{\bm{B}}(\hat{\bm{\Sigma}}),\bm{B},W)\leq \text{\em
ADR}(\hat{\bm{B}}_1,\bm{B};W)$ iff  $\text{\em
ADR}(\hat{\bm{B}}_1,\bm{B};W)\big /\text{\em
ADR}(\tilde{\bm{B}}(\hat{\bm{\Sigma}}),\bm{B},W)\geqslant1$. This
ratio is known as the mean squares relative efficiency (RE). In
presenting the simulation results, we compare the estimators by
using the RE.
\end{remark}

\section{Concluding Remarks}\label{sec:conclusion}
In this paper, we study the asymptotic properties of the UE and the
restricted estimators of the regression coefficients of multivariate
regression model with measurement errors, when the coefficients may
satisfy some restrictions. In comparison with the findings in
literature, we generalize Proposition~A.10 and Corollary~A.2 in Chen
and Nkurunziza~(2016). Further, we generalize Theorem~4.1 of
Jain~{\em et al.}~(2011) in three ways. First, we derive the joint
asymptotic distribution between the UE and any member of the class
of the restricted estimators under the restriction. Recall that, in
the quoted paper, only the marginal asymptotic normality is derived
under the restriction. Second,  we derive the joint asymptotic
normality between the UE and any member of the class of the
restricted estimators under the sequence of local alternative
restrictions. Third, we establish the joint asymptotic distribution
between the UE and the three restricted estimators, given in
Jain~{\em et al.}~(2011), under the restriction and under the
sequence of local alternative restrictions. Further, we establish
the ADR of the UE and the ADR of any member of the class of
restricted estimators under the sequence of local alternative
restrictions. We also study the risk analysis and establish that the
restricted estimators perform better than the unrestricted estimator
in the neighborhood of the restriction. 

\section*{Acknowledgement}
The authors would like to acknowledge the financial support received from
the Natural Sciences and Engineering Research Council of Canada (NSERC).
\appendix
\section{Some technical results}

In this appendix, we give technical results and proofs which are
underlying the established results.
The following lemma is useful in establishing the asymptotic
distributions.
\begin{lem}\label{lemmanorm}
 Let $Y$ be a $p\times q$ random matrix and $\bm{Y}\sim \mathcal{N}_{p\times q}(O,\bm{\Lambda})$, with $\bm{\Lambda}$ a $pq\times pq$ matrix.
 For $j=1,2, \dots, m$, let $\kappa_{j}$ and $\alpha_{j}$ be $p\times p-$ nonrandom matrices, let $\iota_{j}$ and $\beta_{j}$ be
 $q\times q$-nonrandom matrices, and let $\varrho_{j}$ be
 $p\times q$-nonrandom matrices.
Then

$
\left( \begin{array}{ccc}
\kappa_{1}\bm{Y}\iota_{1}+\alpha_1\bm{Y}\beta_1+\varrho_1\\
\kappa_{2}\bm{Y}\iota_{2}+\alpha_2\bm{Y}\beta_2+\varrho_2\\
\vdots \\
\kappa_{m}\bm{Y}\iota_{m}+\alpha_m \bm{Y}\beta_m+\varrho_m
\end{array} \right)
\sim \mathcal{N}_{mq\times p}$ $ \left( \left(\begin{array}{c}
$$\varrho_1$$ \\
$$\varrho_2$$\\
\vdots\\
$$\varrho_{m}$$
\end{array} \right),
\left( \begin{array}{cccc}
\bm{A}_{11} & \bm{A}_{12}&\cdots& \bm{A}_{1m} \\
\bm{A}_{21} & \bm{A}_{22}& \cdots& \bm{A}_{2m}\\
\vdots & \cdots & \cdots & \vdots\\
\bm{A}_{m1} & \bm{A}_{m2}& \cdots& \bm{A}_{mm}
\end{array}
\right) \right),$ where $
 \quad{ } \bm{A}_{ji}=(\bm{A}_{ij})^{'}, i,j=1,2,\dots,m,$ and
\begin{eqnarray*}
& &
\bm{A}_{ij}=(\kappa_{i}\otimes\iota'_{i})\Lambda(\kappa'_{j}\otimes\iota_{j})+(\kappa_{i}\otimes\iota'_{i})\Lambda(\alpha'_j\otimes
\beta_j)
 +(\alpha_i\otimes \beta'_i)\bm{\Lambda}(\alpha^{'}_j\otimes \beta_j)\\&& \quad{ }\quad{ }\quad{ } +(\alpha_i\otimes \beta^{'}_i)\Lambda(\alpha^{'}_j\otimes \beta_j).
\end{eqnarray*}
\end{lem}
\begin{proof}
We have\\
 vec$ \left(\left( \begin{array}{c}
\kappa_{1}\bm{Y}\iota_{1}+\alpha_1\bm{Y}\beta_1+\varrho_1\\
\kappa_{2}\bm{Y}\iota_{2}+\alpha_2\bm{Y}\beta_2+\varrho_2\\
\vdots \\
\kappa_{m}\bm{Y}\iota_{m}+\alpha_m\bm{Y}\beta_m+\varrho_m
\end{array} \right)\right)= \left(\begin{array}{c}
\kappa_{1}\otimes \iota'_{1}+\alpha_1\otimes \beta^{'}_1\\
\kappa_{2}\otimes \iota'_{2}+\alpha_2\otimes \beta^{'}_2\\
\vdots\\
 \kappa_{m}\otimes \iota'_{m}+\alpha_m\otimes \beta^{'}_m
\end{array} \right)\textrm{vec}(\bm{Y})+\left( \begin{array}{c}
\text{vec}(\varrho_1)\\
\text{vec}(\varrho_2)\\
\vdots\\
\text{vec}(\varrho_m)
\end{array} \right)
$ then the rest of the proof follows from the properties of normal
random vectors along with some algebraic computations,  this
completes the proof.
\end{proof}
Note that this result is more general than Corollary~A.2 in Chen and
Nkurunziza~(2016). By using this lemma, we establish the following
lemma, which is more general than Proposition~A.10 and Corollary~A.2
in Chen and Nkurunziza~(2016).The established lemma is particularly
useful in deriving the joint asymptotic normality between
$\hat{\bm{B}}_1$, $\hat{\bm{B}}_2$,$\hat{\bm{B}}_3$ and
$\hat{\bm{B}}_4$.

\begin{lem}\label{lemmajointnorm}
 For $j=1,2,\dots,m$, let $\{\kappa_{jn}\}_{n=1}^{\infty}$, $\{\iota_{jn}\}_{n=1}^{\infty}$
$\{\alpha_{jn}\}_{n=1}^{\infty}$,$\{\beta_{jn}\}_{n=1}^{\infty}$,
$\{\varrho_{jn}\}_{n=1}^{\infty}$,  be sequences of random matrices
such that $\kappa_{jn}\xrightarrow[n\rightarrow
\infty]{P}\kappa_{j}$, $\iota_{jn}\xrightarrow[n\rightarrow
\infty]{P}\iota_{j}$, $\alpha_{jn}\xrightarrow[n\rightarrow
\infty]{P}\alpha_j$, $\beta_{jn}\xrightarrow[n\rightarrow
\infty]{P}\beta_j$, $\varrho_{jn}\xrightarrow[n\rightarrow
\infty]{P}\varrho_j$,  where, for $j=1,2,\dots, m$, $\kappa_j$,
$\alpha_j$, $\iota_j$ and $\beta_j$,$\varrho_j$,  are non-random
matrices as defined in Lemma~\ref{lemmanorm}. If a sequence of
$p\times q$ random matrices $\{\bm{Y}_n\}_{n=1}^{\infty}$ is such
that $\bm{Y}_n\xrightarrow[n\rightarrow \infty]{d}\bm{Y}\sim
\mathcal{N}_{p\times q}(0,\Lambda)$,
where $\bm{\Lambda}$ is a $pq\times pq$ matrix. We have\\
$ \left( \begin{array}{ccc}
\kappa_{1n}\bm{Y}_n\iota_{1n}+\alpha_{1n}\bm{Y}_n\beta_{1n}+\varrho_{1n}\\
\kappa_{2n}\bm{Y}_n\iota_{2n}+\alpha_{2n}\bm{Y}_n\beta_{2n}+\varrho_{2n}\\
\vdots\\
\kappa_{mn}\bm{Y}_n\iota_{mn}+\alpha_{mn}\bm{Y}_n\beta_{mn}+\varrho_{mn}
\end{array} \right)
$$\xrightarrow[n\rightarrow\infty]{d} \bm{U}\sim \mathcal{N}_{mq\times p}
\left(\bm{\varrho},\,\bm{A} \right) $
\\
with $\bm{\varrho}= \left(\begin{array}{c}
$$\varrho_1$$ \\
$$\varrho_2$$\\
\vdots\\
$$\varrho_m$$
\end{array} \right)$,\, $\bm{A}=
\left( \begin{array}{c c c c}
\bm{A}_{11} & \bm{A}_{12}& \cdots & \bm{A}_{1m}\\
\bm{A}_{21} & \bm{A}_{22}& \cdots & \bm{A}_{2m}\\
\vdots & \cdots & \cdots & \vdots\\
\bm{A}_{m1} & \bm{A}_{m2}& \cdots & \bm{A}_{mm}
\end{array}
\right)
$,\\
where $\bm{A}_{ij}$, $i=1,2, \dots, m; j=1,2, \dots, m$ are as
defined in Lemma~\ref{lemmanorm}.
\end{lem}
\begin{proof}
We have
 \begin{eqnarray*}
 \left( \begin{array}{c}
\text{vec}(\kappa_{1n}\bm{Y}_{n}\iota_{1n}+\alpha_{1n}\bm{Y}_{n}\beta_{1n}+\varrho_{1n})\\
\text{vec}(\kappa_{2n}\bm{Y}_{n}\iota_{2n}+\alpha_{2n}\bm{Y}_{n}\beta_{2n}+\varrho_{2n})\\
\vdots\\
\text{vec}(\kappa_{mn}\bm{Y}_{n}\iota_{mn}+\alpha_{mn}\bm{Y}_{n}\beta_{mn}+\varrho_{mn})
\end{array}
\right) = \left( \begin{array}{c}
\kappa_{1n}\otimes \iota'_{1n}+\alpha_{1n}\otimes \beta^{'}_{1n}\\
\kappa_{2n}\otimes \iota'_{2n}+\alpha_{2n}\otimes \beta^{'}_{2n}\\
\vdots\\
 \kappa_{mn}\otimes \iota'_{mn}+\alpha_{mn}\otimes \beta^{'}_{mn}
\end{array} \right)\text{vec}(\bm{Y}_n)\\+
\left( \begin{array}{c}
\text{vec}(\varrho_{1n})\\
\text{vec}(\varrho_{2n})\\
\vdots\\
\text{vec}(\varrho_{mn})
\end{array} \right),
\end{eqnarray*}
where $\text{vec}(\bm{Y}_n)\xrightarrow[n\rightarrow
\infty]{d}\text{vec}(\bm{Y})\sim \mathcal{N}_{pq}(0,\bm{\Lambda})$,
$\left(
\begin{array}{c}
\text{vec}(\varrho_{1n})\\
\text{vec}(\varrho_{2n})\\
\vdots\\
\text{vec}(\varrho_{mn})
\end{array} \right)\xrightarrow[n\rightarrow\infty]{P}\left( \begin{array}{c}
\text{vec}(\varrho_{1})\\
\text{vec}(\varrho_2)\\
\vdots\\
\text{vec}(\varrho_{m})
\end{array} \right), $ \\and  $\left(
\begin{array}{c}
\kappa_{1n}\otimes \iota'_{1n}+\alpha_{1n}\otimes \beta^{'}_{1n}\\
\kappa_{2n}\otimes \iota'_{2n}+\alpha_{2n}\otimes \beta^{'}_{2n}\\
\vdots\\
 \kappa_{mn}\otimes \iota'_{mn}+\alpha_{mn}\otimes \beta^{'}_{mn}
\end{array} \right)
\xrightarrow[n\rightarrow\infty]{P} \left( \begin{array}{c}
\kappa_{1}\otimes \iota'_{1}+\alpha_1\otimes \beta^{'}_1\\
\kappa_{2}\otimes \iota'_{2}+\alpha_2\otimes \beta^{'}_2\\
\vdots\\
 \kappa_{m}\otimes \iota'_{m}+\alpha_m\otimes \beta^{'}_m
\end{array} \right).$\\
Then, by using Slutsky's theorem, we have \\
vec$ \left( \begin{array}{ccc}
\kappa_{1n}\bm{Y}_{n}\iota_{1n}+\alpha_{1n}\bm{Y}_{n}\beta_{1n}+\varrho_{1n}\\
\kappa_{2n}\bm{Y}_{n}\iota_{2n}+\alpha_{2n}\bm{Y}_{n}\beta_{2n}+\varrho_{2n}\\
\vdots\\
\kappa_{mn}\bm{Y}_{n}\iota_{mn}+\alpha_{mn}\bm{Y}_{n}\beta_{mn}+\varrho_{mn}
\end{array} \right)
$ $\xrightarrow[n\rightarrow\infty]{d}\text{vec} \left(
\begin{array}{ccc}
\kappa_{1}\bm{Y}\iota_{1}+\alpha_1\bm{Y}\beta_1+\varrho_1\\
\kappa_{2}\bm{Y}\iota_{2}+\alpha_2\bm{Y}\beta_2+\varrho_2\\
\vdots \\
\kappa_{m}\bm{Y}\iota_{m}+\alpha_m\bm{Y}\beta_m+\varrho_m
\end{array} \right)$ and then\\
$ \left( \begin{array}{ccc}
\kappa_{1n}\bm{Y}_{n}\iota_{1n}+\alpha_{1n}\bm{Y}_{n}\beta_{1n}+\varrho_{1n}\\
\kappa_{2n}\bm{Y}_{n}\iota_{2n}+\alpha_{2n}\bm{Y}_{n}\beta_{2n}+\varrho_{2n}\\
\vdots\\
\kappa_{mn}\bm{Y}_{n}\iota_{mn}+\alpha_{mn}\bm{Y}_{n}\beta_{mn}+\varrho_{mn}
\end{array} \right)
$$\xrightarrow[n\rightarrow\infty]{d}
\left( \begin{array}{ccc}
\kappa_{1}\bm{Y}\iota_{1}+\alpha_1\bm{Y}\beta_1+\varrho_1\\
\kappa_{2}\bm{Y}\iota_{2}+\alpha_2\bm{Y}\beta_2+\varrho_2\\
\vdots \\
\kappa_{m}\bm{Y}\iota_{m}+\alpha_m\bm{Y}\beta_m+\varrho_m
\end{array} \right)\equiv \bm{U}$.\\ Then, the proof follows directly from Lemma~\ref{lemmanorm}. 
\end{proof}
 From this lemma, we establish the following corollary.
\begin{corol}\label{corollnorm}
Suppose that the conditions Lemma~\ref{lemmajointnorm} hold. We
have\\
$(\bm{Y}^{'}_n,(\bm{Y}_n+\alpha_{2n}\bm{Y}_n\beta_{2n}+\varrho_{2n})^{'})^{'}\xrightarrow[n\rightarrow\infty]{d}(\bm{Y}^{'},(\bm{Y}+\alpha_{2}
\bm{Y}\beta_{2}+\varrho_{2})^{'})^{'}$, with
\begin{eqnarray*}
\left( \begin{array}{ccc}
\bm{Y}\\
\bm{Y}+\alpha_{2} \bm{Y}\beta_{2}+\varrho_{2}
\end{array} \right)\sim \mathcal{N}_{2q\times p}
\left( \left(\begin{array}{c}
0 \\
\varrho_{2}
\end{array} \right),
\left( \begin{array}{c c}
\bm{V}_{11} & \bm{V}_{12} \\
\bm{V}_{21} & \bm{V}_{22}
\end{array}
\right)
\right)
\end{eqnarray*}
where $\bm{V}_{11}=\bm{\Lambda}$;\quad{ }
$\bm{V}_{12}=\bm{\Lambda}+\bm{\Lambda} (\alpha^{'}_{2}\otimes
\beta_{2})$;\quad{ }  $\bm{V}_{21}=(\bm{V}_{12})^{'}$;\\
$\bm{V}_{22}=(\bm{I}_{pq}+\alpha_{2}\otimes
\beta^{'}_{2})\bm{\Lambda}(\bm{I}_{pq}+\alpha^{'}_{2}\otimes
\beta_{2}).$
\end{corol}
The proof follows directly from Lemma~\ref{lemmajointnorm} by taking
$m=2$, $\kappa_{jn}=\bm{I}_{p}$, $\iota_{jn}=\bm{I}_{q}$, $\alpha_{1n}=0$, $\beta_{1n}=0$ and $\varrho_{1n}=0$. %

\begin{proof}[Proof of Theorem~\ref{propjointclassloc}]
We have
\begin{eqnarray*}
(\hat{\bm{B}}(\hat{\bm{\Sigma}})-\bm{B})=(\hat{\bm{B}}_1-\bm{B})+(\hat{\bm{\Sigma}})^{-1}{\bm{R}^{'}_1}[\bm{R}_{1}(\hat{\bm{\Sigma}})^{-1}{\bm{R}^{'}_1}]^{-1}
(\theta-\bm{R}_{1}\hat{\bm{B}}_1\bm{R}_{2})({\bm{R}^{'}_{2}}\bm{R}_{2})^{-1}{\bm{R}^{'}_{2}}\\
=(\hat{\bm{B}}_1-\bm{B})-\bm{G}_{2n}(W_{0})\bm{R}_{1}\left(\hat{\bm{B}}_1-\bm{B}\right)\bm{R}_{2}\bm{P}_{n}+\bm{G}_{2n}(W_{0})(\bm{\theta}-\bm{R}_{1}B_1\bm{R}_{2})\bm{P}_{n}.
\end{eqnarray*}
with
$\bm{G}_{2n}(W_{0})=(\hat{\bm{\Sigma}})^{-1}{\bm{R}^{'}_1}[\bm{R}_{1}(\hat{\bm{\Sigma}})^{-1}{\bm{R}^{'}_1}]^{-1}$
and $\bm{P}_{n}=({\bm{R}^{'}_{2}}\bm{R}_{2})^{-1}{\bm{R}^{'}_{2}}$.
Then, since
\\$\bm{R}_{1}\bm{B}\bm{R}_2=\bm{\theta}+\bm{\theta}_{0}\big/\sqrt{n}$, this last relation
gives
\begin{eqnarray*}
n^{\frac{1}{2}}(\hat{\bm{B}}(\hat{\bm{\Sigma}})-\bm{B})=n^{\frac{1}{2}}(\hat{\bm{B}}_1-\bm{B})-
\bm{G}_{2n}(W_{0})\bm{R}_{1}\left(n^{\frac{1}{2}}(\hat{\bm{B}}_1-\bm{B})\right)\bm{R}_{2}\bm{P}_{n}-\bm{G}_{2n}(W_{0})\theta_{0}\bm{P}_{n}.
\end{eqnarray*}
Hence,
\begin{eqnarray*}
\left( \begin{array}{c}
\sqrt{n}(\hat{\bm{B}}_1-\bm{B})\\
\sqrt{n}(\hat{\bm{B}}(\hat{\bm{\Sigma}})-\bm{B})
\end{array} \right)
= \left( \begin{array}{c}
n^{\frac{1}{2}}(\hat{\bm{B}}_1-\bm{B})\\
n^{\frac{1}{2}}(\hat{\bm{B}}_1-\bm{B})-\bm{G}_{2n}(W_{0})\bm{R}_{1}\left(n^{\frac{1}{2}}(\hat{\bm{B}}_1-\bm{B})\right)\bm{R}_{2}\bm{P}_{n}
\end{array} \right)\\
+\left(
   \begin{array}{c}
     \bm{0} \\
     -\bm{G}_{2n}(W_{0})\bm{\theta}_{0}\bm{P}_{n} \\
   \end{array}
 \right).
\end{eqnarray*}
Note that $\bm{G}_{2n}\xrightarrow[n\rightarrow\infty]{P}
\bm{G}_2(\bm{Q}_{0})$ and
$\bm{P}_{n}\xrightarrow[n\rightarrow\infty]{P}P$, with
$\bm{G}_2(\bm{Q}_{0})=(\bm{Q}_{0})^{-1}\bm{R}^{'}_1(\bm{R}_{1}(\bm{Q}_{0})^{-1}\bm{R}^{'}_1)^{-1}$
and $\bm{P}=({\bm{R}^{'}_{2}}\bm{R}_{2})^{-1}{\bm{R}^{'}_{2}}$.
Further, let $\bm{\alpha}=\bm{G}_2(\bm{Q}_{0})\bm{R}_{1}$, let
$\bm{\beta}=\bm{R}_{2}P$ and let\\
$\bm{\varrho}=\bm{G}_{2}(\bm{Q}_{0})\bm{\theta}_{0}\bm{P}$. By using
Corollary~\ref{corollnorm}, we have
\begin{eqnarray*}
\left( \begin{array}{ccc}
\sqrt{n}(\hat{\bm{B}}_1-\bm{B})\\
\sqrt{n}(\hat{\bm{B}}(\hat{\bm{\Sigma}})-\bm{B})
\end{array} \right)&
\xrightarrow[n\rightarrow\infty]{d}& \left( \begin{array}{ccc}
\bm{Y}\\
\bm{Y}+\bm{\alpha} \bm{Y}\bm{\beta}+\bm{\varrho}
\end{array} \right)\\\, & & \sim \mathcal{N}_{2q\times p}
\left( \left(\begin{array}{c}
\bm{0} \\
\bm{\mu}\left(\bm{Q}_{0}\right)
\end{array} \right),
\left( \begin{array}{c c}
\bm{\Sigma}_{11} & \bm{\Sigma}_{12}\left(\bm{Q}_{0}\right) \\
\bm{\Sigma}_{21}\left(\bm{Q}_{0}\right) &
\bm{\Sigma}_{22}\left(\bm{Q}_{0}\right)
\end{array}
\right) \right),
\end{eqnarray*}
this completes the proof.
\end{proof}
\begin{proof}[Proof of Theorem~\ref{propjointclass}]
The proof follows from Theorem~\ref{propjointclassloc} by taking
$\bm{\mu}\left(\bm{Q}_{0}\right)=\bm{0}$.
\end{proof}
\begin{proof}[Proof of Theorem~\ref{propjoinnorm}]
 From (\ref{B2hat}), \eqref{B3} and \eqref{B4}, we have
\begin{eqnarray*}
n^{\frac{1}{2}}(\hat{\bm{B}}_2-\bm{B})=n^{\frac{1}{2}}(\hat{\bm{B}}_1-\bm{B})-n^{\frac{1}{2}}\bm{G}_{2n}\bm{R}_{1}(\hat{\bm{B}}_1-\bm{B})\bm{R}_{2}\bm{P}_{n}
+n^{\frac{1}{2}}\bm{G}_{2n}(\theta-\bm{R}_{1}\bm{B}\bm{R}_2)\bm{P}_{n}\\
n^{\frac{1}{2}}(\hat{\bm{B}}_3-\bm{B})=n^{\frac{1}{2}}(\hat{\bm{B}}_1-\bm{B})+n^{\frac{1}{2}}\bm{G}_{3n}(\theta-\bm{R}_{1}\hat{\bm{B}}_1\bm{R}_{2})\bm{P}_{n} \\
n^{\frac{1}{2}}(\hat{\bm{B}}_4-\bm{B})=n^{\frac{1}{2}}(\hat{\bm{B}}_1-\bm{B})-n^{\frac{1}{2}}\bm{G}_{4n}\bm{R}_{1}(\hat{\bm{B}}_1-\bm{B})\bm{R}_{2}\bm{P}_{n}
+n^{\frac{1}{2}}\bm{G}_{4n}(\theta-\bm{R}_{1}\bm{B}\bm{R}_2)\bm{P}_{n}.
\end{eqnarray*}
with
$\bm{G}_{2n}=(\bm{S} \bm{K}_{X})^{-1}{\bm{R}^{'}_1}[\bm{R}_{1}(\bm{S} \bm{K}_{X})^{-1}{\bm{R}^{'}_1}]^{-1}$, \, $\bm{S}=\bm{X}'\bm{X}$,  $\bm{G}_{3n}=\bm{S}^{-1}{\bm{R}^{'}_1}[\bm{R}_{1}\bm{S}^{-1}{\bm{R}^{'}_1}]^{-1}$,\\
$\bm{G}_{4n}={\bm{R}^{'}_1}[\bm{R}_{1}{\bm{R}^{'}_1}]^{-1}$, \,
$\bm{P}_{n}=(\bm{R}_{2}^{'}\bm{R}_{2})^{-1}\bm{R}_{2}^{'}$.
Then, since
$\bm{R}_{1}\bm{B}\bm{R}_2=\bm{\theta}+\bm{\theta}_{0}/\sqrt{n}$, we
have
\begin{eqnarray*}
n^{\frac{1}{2}}(\hat{\bm{B}}_2-\bm{B})=n^{\frac{1}{2}}(\hat{\bm{B}}_1-\bm{B})-n^{\frac{1}{2}}\bm{G}_{2n}\bm{R}_{1}(\hat{\bm{B}}_1-\bm{B})\bm{R}_{2}\bm{P}_{n}
+\bm{G}_{2n}\theta_{0}\bm{P}_{n}\\
n^{\frac{1}{2}}(\hat{\bm{B}}_3-\bm{B})=n^{\frac{1}{2}}(\hat{\bm{B}}_1-\bm{B})-n^{\frac{1}{2}}\bm{G}_{3n}\bm{R}_{1}(\hat{\bm{B}}_1-\bm{B})\bm{R}_{2}\bm{P}_{n}
+\bm{G}_{3n}\theta_{0}\bm{P}_{n}. \\
n^{\frac{1}{2}}(\hat{\bm{B}}_4-\bm{B})=n^{\frac{1}{2}}(\hat{\bm{B}}_1-\bm{B})-n^{\frac{1}{2}}\bm{G}_{4n}\bm{R}_{1}(\hat{\bm{B}}_1-\bm{B})\bm{R}_{2}\bm{P}_{n}
+\bm{G}_{4n}\theta_{0}\bm{P}_{n}.
\end{eqnarray*}
Therefore,
\begin{eqnarray*}
\left(
  \begin{array}{c}
    n^{\frac{1}{2}}(\hat{\bm{B}}_1-\bm{B}) \\
    n^{\frac{1}{2}}(\hat{\bm{B}}_2-\bm{B}) \\
    n^{\frac{1}{2}}(\hat{\bm{B}}_3-\bm{B}) \\
    n^{\frac{1}{2}}(\hat{\bm{B}}_4-\bm{B}) \\
  \end{array}
\right)=\left(
          \begin{array}{c}
            n^{\frac{1}{2}}(\hat{\bm{B}}_1-\bm{B}) \\
            n^{\frac{1}{2}}(\hat{\bm{B}}_1-\bm{B})-\bm{G}_{2n}\bm{R}_{1}\left(n^{\frac{1}{2}}(\hat{\bm{B}}_1-\bm{B})\right)\bm{R}_{2}\bm{P}_{n}
+\bm{G}_{2n}\theta_{0}\bm{P}_{n} \\
            n^{\frac{1}{2}}(\hat{\bm{B}}_1-\bm{B})-\bm{G}_{3n}\bm{R}_{1}\left(n^{\frac{1}{2}}(\hat{\bm{B}}_1-\bm{B})\right)\bm{R}_{2}\bm{P}_{n}
+\bm{G}_{3n}\theta_{0}\bm{P}_{n} \\
            n^{\frac{1}{2}}(\hat{\bm{B}}_1-\bm{B})-\bm{G}_{4n}\bm{R}_{1}\left(n^{\frac{1}{2}}(\hat{\bm{B}}_1-\bm{B})\right)\bm{R}_{2}\bm{P}_{n}
+\bm{G}_{4n}\theta_{0}\bm{P}_{n} \\
          \end{array}
        \right),
\end{eqnarray*}
with  $n^{\frac{1}{2}}(\hat{\bm{B}}_1-\bm{B})
\xrightarrow[n\rightarrow\infty]{d}\eta_1\sim \mathcal{N}_{p \times
q}\left(\bm{O}, \bm{\Sigma}_{11}\right)$,\\ $\bm{G}_{2n}
\xrightarrow[n\rightarrow\infty]{P}\bm{G}_{2}=(\bm{\Sigma}
\bm{K})^{-1}\bm{R}^{'}_1(\bm{R}_{1}(\bm{\Sigma}
\bm{K})^{-1}\bm{R}^{'}_1)^{-1}$, $\bm{G}_{3n}
\xrightarrow[n\rightarrow\infty]{P}\bm{G}_{3}=\bm{\Sigma}^{-1}\bm{R}^{'}_1(\bm{R}_{1}\bm{\Sigma}^{-1}\bm{R}^{'}_1)^{-1}$,
$\bm{G}_{4n}
\xrightarrow[n\rightarrow\infty]{P}\bm{G}_{4}=\bm{R}^{'}_1(\bm{R}_{1}
\bm{R}^{'}_1)^{-1}$, $\bm{P}_{n}
\xrightarrow[n\rightarrow\infty]{P}\bm{P}=(\bm{R}_{2}^{'}\bm{R}_{2})^{-1}\bm{R}_{2}^{'}$.
Therefore, by using Lemma~\ref{lemmajointnorm}, we get the statement
of the proposition.
\end{proof}

\begin{proof}[Proof of Theorem~\ref{propADR}]
 The first statement follows from Theorem~\ref{propjointclassloc}. Further, we have,
\begin{eqnarray*}
\text{ADR}\left(\tilde{\bm{B}}(\hat{\bm{\Sigma}}),\bm{B},W\right)=\text{tr}\left[(W\otimes \bm{I}_{q})\textrm{E}\left[\textrm{vec}(\bm{\eta}^{*})\left(\textrm{vec}(\bm{\eta}^{*})\right)'\right]\right]\\
=\text{tr}((W\otimes \bm{I}_{q})
(\bm{\Sigma}_{22}(\bm{Q}_{0})))+\text{tr}\left(\bm{\mu}'(\bm{Q}_{0})W\bm{\mu}(\bm{Q}_{0})\right).
\end{eqnarray*}
This gives
\begin{eqnarray*}
\text{ADR}(\tilde{\bm{B}}(\hat{\bm{\Sigma}}),\bm{B},W)=\text{ADR}(\hat{\bm{B}}_1,\bm{B};W)-\text{tr}((W\otimes \bm{I}_{q})(A_1\Lambda(J^{'}_1(\bm{Q}_{0})\otimes J))\\
-\text{tr}((W\otimes \bm{I}_{q})((\bm{J}_1(\bm{Q}_{0})\otimes \bm{J})\bm{\Lambda}  \bm{A}'_{1})+\text{tr}((W\otimes \bm{I}_{q})((\bm{J}_1(\bm{Q}_{0})\otimes \bm{J})\bm{\Lambda}(J^{'}_1(\bm{Q}_{0})\otimes J)))\\
+\text{tr}((\bm{C}_4\bm{C}^{'}_4)\otimes(\bm{C}^{'}_3(\bm{Q}_{0})W\bm{C}_3(\bm{Q}_{0}))\text{vec}(\bm{\theta}_0)(\text{vec}(\bm{\theta}_0))^{'}).
\end{eqnarray*}
Further, one can verify that
$\text{vec}(\bm{J}_1(\bm{Q}_{0})\otimes \bm{J})=\text{vec}((\Sigma
K)^{-1}\otimes \bm{I}_q)$. Then,
the rest of the proof follows from some algebraic computations.
\end{proof}
\begin{proof}[Proof of Theorem~\ref{compareADR}]
From Theorem~\ref{propADR}, we have 
\begin{eqnarray}
\text{ADR}(\tilde{\bm{B}}(\hat{\bm{\Sigma}}),\bm{B},W)
=\text{ADR}(\hat{\bm{B}}_1,\bm{B},W)-f_1(\bm{Q}_{0})+(\text{vec}(\bm{\theta}_0))^{'}\bm{F}_1(\bm{Q}_{0})\text{vec}(\bm{\theta}_0).
\label{prelADR}
\end{eqnarray}
Note that $f_1(\bm{Q}_{0})\geqslant0$ and obviously, if
$f_1(\bm{Q}_{0})=0$, ADR$(\tilde{\bm{B}}(\hat{\bm{\Sigma}}), B;
W)>$ADR$(\hat{\bm{B}}_1, \bm{B},W)$ provided that
$(\text{vec}(\bm{\theta}_0))^{'}\bm{F}_1(\bm{Q}_{0})\text{vec}(\bm{\theta}_0)>0$.
Thus, we only consider the case where $f_1(\bm{Q}_{0})>0$.
From~\eqref{prelADR}, ADR$(\tilde{\bm{B}}(\hat{\bm{\Sigma}}), B;
W)>$ADR$(\hat{\bm{B}}_1, \bm{B},W)$ if and only if\\
$-f_1(\bm{Q}_{0})+(\text{vec}(\bm{\theta}_0))^{'}\bm{F}_1(\bm{Q}_{0})\text{vec}(\bm{\theta}_0)>0$.
If
$f_1(\bm{Q}_{0})<(\text{vec}(\bm{\theta}_0))^{'}\bm{F}_1(\bm{Q}_{0})\text{vec}(\bm{\theta}_0),$
we have
$$\frac{f_1(\bm{Q}_{0})}{(\text{vec}(\bm{\theta}_0))^{'}\bm{F}_1(\bm{Q}_{0})\text{vec}(\bm{\theta}_0)}<1, \, \mbox{
and } \,
(\text{vec}(\bm{\theta}_0))^{'}\text{vec}(\bm{\theta}_0)\displaystyle\frac{f_1(\bm{Q}_{0})}{(\text{vec}(\bm{\theta}_0))^{'}
\bm{F}_1(\bm{Q}_{0})\text{vec}(\bm{\theta}_0)}<||\bm{\theta}_0||^2.
$$
Further, since $f_1(\bm{Q}_{0})>0$, we have
\begin{equation}\label{prelADRbis}
\displaystyle\frac{(\text{vec}(\bm{\theta}_0))^{'}\bm{F}_1(\bm{Q}_{0})\text{vec}(\bm{\theta}_0)}{(\text{vec}(\bm{\theta}_0))^{'}\text{vec}(\bm{\theta}_0)f_1}>\frac{1}{||\bm{\theta}_0||^2}.
\end{equation}
Further, by using Courant Theorem, we have
\begin{equation} \label{courant}
\displaystyle
\textrm{ch}_{min}(\bm{F}_1(\bm{Q}_{0}))<\displaystyle\frac{(\text{vec}(\bm{\theta}_0))^{'}\bm{F}_1(\bm{Q}_{0})\text{vec}(\bm{\theta}_0)}
{(\text{vec}(\bm{\theta}_0))^{'}\text{vec}(\bm{\theta}_0)}<\displaystyle
\textrm{ch}_{max}(\bm{F}_1(\bm{Q}_{0})).
\end{equation}
Therefore, for the inequality in (\ref{prelADRbis}) to hold, it
suffices to have
$$\frac{1}{||\bm{\theta}_0||^2}<\frac{\textrm{ch}_{min}(\bm{F}_1(\bm{Q}_{0}))}{f_1(\bm{Q}_{0})}.$$
That is if
$||\bm{\theta}_0||^2>\displaystyle\frac{f_1(\bm{Q}_{0})}{\textrm{ch}_{min}(\bm{F}_1(\bm{Q}_{0}))}$,
we have
$\text{ADR}(\tilde{\bm{B}}(\hat{\bm{\Sigma}}),\bm{B},W)>\text{ADR}(\hat{\bm{B}}_1,\bm{B};W).$
Further, if
$f_1(\bm{Q}_{0})>(\text{vec}(\bm{\theta}_0))^{'}\bm{F}_1(\bm{Q}_{0})\text{vec}(\bm{\theta}_0)$,
by using (\ref{courant}), we establish the condition that %
if
$||\bm{\theta}_0||^2<\displaystyle\frac{f_1(\bm{Q}_{0})}{\textrm{ch}_{max}(\bm{F}_1(\bm{Q}_{0}))}$,
then
$\text{ADR}(\tilde{\bm{B}}(\hat{\bm{\Sigma}}),\bm{B},W)<\text{ADR}(\hat{\bm{B}}_1,\bm{B};W),$
this completes the proof.
\end{proof}


\end{document}